\documentclass [11pt]{article}
\usepackage{amssymb,amsmath,amsthm}

\def\N{\mathbb{N}}

\def\Z{\mathbb{Z}}

\def\F{\mathbb{F}}

\newtheorem{theorem}{Theorem}[section]
\newtheorem{proposition}[theorem]{Proposition}
\newtheorem{corollary}[theorem]{Corollary}
\newtheorem{lemma}[theorem]{Lemma}
\newtheorem{definition}[theorem]{Definition}

\newtheorem{conjecture}[theorem]{Conjecture}

\begin{document}
\title{All bi-unitary perfect polynomials over $\F_2$ with only Mersenne primes as odd divisors}
 \author{Luis H. Gallardo - Olivier Rahavandrainy \\
Univ. Brest, UMR CNRS 6205\\
Laboratoire de Math\'ematiques de Bretagne Atlantique\\
e-mail: luis.gallardo@univ-brest.fr - olivier.rahavandrainy@univ-brest.fr}
\maketitle
\begin{itemize}
\item[a)]
\item[b)]
\end{itemize}

{\bf{Abstract}}

The paper is about  a polynomial variant of a classical arithmetic problem. More precisely,
we obtain all non-splitting bi-unitary perfect polynomials
over the prime field of two elements,
which have only Mersenne polynomials as odd irreducible divisors.
{\section{Introduction}}
We consider the following functions and notions over $\F_2[x]$, as generalizations of the integer case (see \cite{Yamada} and references therein). Let $S \in \F_2[x]$ be a nonzero polynomial. A divisor $D$ of $S$ is called {\emph{unitary}} if $\gcd(D,S/D)=1$. It is {\emph{bi-unitary}} if $\gcd_u(D,S/D)=1$, where $\gcd_u(S,T)$  is the greatest common unitary divisor of $S$ and $T$. We let  $\sigma(S)$ (resp. $\sigma^*(S)$, $\sigma^{**}(S)$) denote the sum of all divisors
(resp. unitary divisors, bi-unitary divisors) of $S$. As usual, $\omega(S)$ denotes the number of distinct irreducible factors of $S$. The functions $\sigma$, $\sigma^*$ and $\sigma^{**}$ are all multiplicative (Definition \ref{defmultipl}). One says \cite{Canaday,BeardU,BeardBiunitary} that $S$ is \emph{perfect} (resp. \emph{unitary perfect}, \emph{bi-unitary perfect}) if $\sigma(S) = S$
(resp. $\sigma^*(S)=S$, $\sigma^{**}(S)=S$). $S$ is \emph{odd} \cite{Gall-Rahav5} if $\gcd(S,x(x+1))=1$ and it is \emph{even}, otherwise.

A {\it{Mersenne} $(prime)$} \cite{Gall-Rahav-mersenn} is a (an irreducible) polynomial of the form $1+x^a(x+1)^b$, where $a$ and $b$ are two coprime positive integers.

Finally, we say that a bi-unitary perfect polynomial (\emph{b.u.p.} polynomial) is \emph{indecomposable} if it is not a product of two coprime non-constant bi-unitary perfect polynomials.

One can find many papers (from $1941$ to the present) studying perfect, (bi-)unitary perfect polynomials over $\F_2$ in the mathematical literature. Canaday \cite{Canaday} and Beard \cite{BeardU}  begun the study of perfect, respectively, unitary perfect polynomials.
Later, Gallardo and Rahavandrainy \cite{Gall-Rahav7,Gall-Rahav5, Gall-Rahav11}, found the list of all
these polynomials with $\omega(A) \leq 4$, and the list of all of them which are divisible only by $x$, $x+1$ and by Mersenne primes \cite{Gall-Rahav12,Gall-Rahav-up-allmersenne}.

We now consider indecomposable bi-unitary perfect polynomials (over $\F_2$)
(say, \emph{i.b.u.p.} polynomials)
with only Mersenne primes as odd divisors.
Beard \cite{BeardBiunitary} discovered thirteen of them: $C_1,\ldots,C_{13}$ (cf. {\bf{Notation}}, there si a misprint for $C_6$, in \cite{BeardBiunitary}). Our main result (Theorem~\ref{casemersenne}) completes that list (with new polynomials: $C_{14},\ldots,C_{23}$).

The paper is a bit technical but elementary. It is organized as follows.
Section \ref{preliminaire} contains the main technical results useful for the proof of Theorem \ref{casemersenne}. Some of them give basic properties of b.u.p. polynomials. Namely,
Corollary \ref{nombreminimal} proves that
if $A$ is a non-constant b.u.p. polynomial, then $x(x+1)$ divides $A$ so that $\omega(A) \geq 2$.  Moreover, Corollary \ref{nombreminimal} together with \cite[Theorem 5]{BeardBiunitary} state that the only
b.u.p. polynomials over $\F_2$ with exactly two distinct prime
divisors are $x^2(x+1)^{2}$ and $x^{2^n-1}(x+1)^{2^n-1}$, for any nonnegative integer $n$.

Note \cite{Wall} that $6, 60$ and $90$ are the only b.u.p numbers.\\
\\
{\bf{Notation}}
\begin{itemize}
\item
The set of integers (resp. of nonnegative
integers, of positive integers) is denoted by $\Z$ (resp. $\N$, $\N\sp{*}$).
\item
For $S, T \in \F_2[x]$ and for $m \in \N^*$, $S^m \mid T$ (resp. $S^m \| T$) means that $S$ divides $T$ (resp. $S^m \mid T$ but $S^{m+1} \nmid T$). We let $\overline{S}$  denote the polynomial defined by $\overline{S}(x) = S(x+1)$.
\item \text{We put}
\end{itemize}
\begin{itemize}
\item[ ] $M_1=1+x+x^2 = \sigma(x^2),\ M_2=1+x+x^3,\ M_3=\overline{M_2}=1+x^2+x^3,$
\item[ ] $M_4=1+x+x^2+x^3+x^4 = \sigma(x^4), M_5=\overline{M_4}=1+x^3+x^4,$
\item[ ] $C_1 =x^3(x+1)^4M_1, C_2=x^3(x+1)^5{M_1}^2, C_3=x^4(x+1)^4{M_1}^2,$
\item[ ] $C_4=x^6(x+1)^6{M_1}^2, \
C_5=x^4(x+1)^5{M_1}^3, \  C_6=x^7(x+1)^8{M_5},$
\item[ ] $C_7=x^7(x+1)^9{M_5}^2, \
C_8 =x^8(x+1)^8M_4M_5, \ C_9 =x^8(x+1)^9M_4{M_5}^2,$
\item[ ] $C_{10}=x^7(x+1)^{10}{M_1}^2M_5, \
C_{11}=x^7(x+1)^{13}{M_2}^2{M_3}^2,$
\item[ ] $C_{12}=x^9(x+1)^{9}{M_4}^2{M_5}^2, \
C_{13}=x^{14}(x+1)^{14}{M_2}^2{M_3}^2$
\end{itemize}
\begin{itemize}
\item[ ] $C_{14}=x^{8}(x+1)^{10}{M_1}^2M_4M_5,\
C_{15}=x^{8}(x+1)^{12}{M_1}^2{M_2}{M_3}M_4,$
\item[ ] $C_{16}=x^{10}(x+1)^{13}{M_1}^2{M_2}^2{M_3}^2M_4, \
C_{17}=x^{13}(x+1)^{13}{M_1}^2{M_2}^4{M_3}^4M_4M_5,$
\item[ ] $C_{18}=x^{12}(x+1)^{13}{M_1}^2{M_2}^3{M_3}^3, \ C_{19}=x^{9}(x+1)^{13}{M_2}^2{M_3}^2{M_4}^2,$
\item[ ] $C_{20}=x^{8}(x+1)^{13}{M_2}^2{M_3}^2{M_4}, \ C_{21}=x^{9}(x+1)^{10}{M_1}^2{M_4}^2{M_5},$
\item[ ] $C_{22}=x^{7}(x+1)^{12}{M_1}^2{M_2}{M_3}, C_{23}=x^{9}(x+1)^{12}{M_1}^2{M_2}{M_3}{M_4}^2.$
\end{itemize}
The polynomials $M_1,\ldots, M_5$ are all Mersenne primes.

\begin{theorem}\label{casemersenne}
Let $A=x^a(x+1)^b {P_1}^{h_1} \cdots {P_r}^{h_r} \in \F_2[x]$ be such that
the $P_j$'s are Mersenne primes, $a,b,h_j \in \N$ and $\omega(A) \geq 3$.
Then $A$ is i.b.u.p. if and only if
$A, \overline{A} \in \{C_j: 1 \leq j \leq 23 \}$.
\end{theorem}

We sketch our method (initiated in \cite{Rahav} and \cite{Gall-Rahav-up-allmersenne}).
In addition to $x$ and $x+1$, determine all possible Mersenne primes which divide such b.u.p. polynomials, give upper bounds of their
exponents without considering several distinct cases. Compute with {\tt{Maple}}, to obtain the list of all such b.u.p. polynomials.

\section{Preliminaries}\label{preliminaire}
We need the following results. Some of them are obvious or (well) known, so we omit
their proofs. We put $\text{${\mathcal{M}}:= \{M_1,M_2,M_3,M_4,M_5\}$}$.
\begin{definition} \label{defmultipl}
{\emph{A nonzero map $f$, from $\F_2[x] \setminus \{0\}$ to $\F_2[x]$,  is {\emph{multiplicative}} if $f(A_1A_2) = f(A_1) f(A_2),$ for all coprime  nonzero polynomials $A_1$ and $A_2$.}}
\end{definition}
\begin{lemma}\label{coprime}
Let $T \in \F_2[x]$ be irreducible and $n \in \N^*$. Then $$\gcd(\sigma(T^n), \sigma(T^{n-1})) = 1.$$
\end{lemma}
\begin{proof}
If $P$ is an irreducible common divisor of $\sigma(T^n)$ and $\sigma(T^{n-1})$, with $P \not= 1$, then $P$ divides $\sigma(T^n) + \sigma(T^{n-1}) = T^n$. Hence, $P = T$ and so $T$ divides $\sigma(T^n)$, which is impossible.
\end{proof}
\begin{lemma}\label{sigmastardegree}
Let $S, T \in \F_2[x] \setminus \{0,1\}$ be such that $\deg(S) = \deg(T)$. Then, $\sigma^{**}(S)=T$ if and only if $T$ divides $\sigma^{**}(S).$
\end{lemma}
\begin{proof} It follows from the fact that $\sigma^{**}$ is degree-preserving.
\end{proof}
\begin{lemma}\label{multiplicativity}
If $A = A_1A_2 \not= 0$ is b.u.p. over $\F_2$ and if
$\gcd(A_1,A_2) =~1$, then $A_1$ is b.u.p. if and only if
$A_2$ is b.u.p..
\end{lemma}
\begin{proof}
Since $A$ is b.u.p. and $\sigma^{**}$ is multiplicative, we get
$$A_1 A_2 = \sigma^{**}(A_1A_2) = \sigma^{**}(A_1) \sigma^{**}(A_2).$$ So,
$\sigma^{**}(A_1) = A_1 \iff  \sigma^{**}(A_1) A_2 = A_1A_2 \iff \sigma^{**}(A_1) A_2 = \sigma^{**}(A_1A_2).$
Thus
$$\sigma^{**}(A_1) = A_1 \iff
 \sigma^{**}(A_1) A_2 = \sigma^{**}(A_1) \sigma^{**}(A_2) \iff A_2 = \sigma^{**}(A_2).$$
\end{proof}
\begin{lemma} [\cite{BeardBiunitary}, Theorem 2]\label{translation}
If $A$ is b.u.p. over $\F_2$, then $\overline{A}$
is also b.u.p.
\end{lemma}
Lemma \ref{complete2} follows from \cite[Lemma 2.6]{Gall-Rahav13} and from Canaday's paper
\cite[Lemmas 4, 5, 6, Theorem 8 and Corollary on page 728]{Canaday}.
\begin{lemma}\label{complete2}
Let $P, Q \in \F_2[x]$ be odd and irreducible, and let $n,m \in \N^*$.
\begin{itemize}
\item[\rm{(i)}] If $P$ is a Mersenne prime, then $\sigma(P^{2n})$ is odd and square-free.
\item[\rm{(ii)}] If $P$ is a Mersenne prime and if $P=P^*$, then $P=M_1$ or $P=M_4$.
\item[\rm{(iii)}]If $\sigma(x^{2n}) = PQ$ and $\overline{P} = \sigma(x^{2m})$, then $n=4$, $m=1$ and $Q = P(x^3)$.
\item[\rm{(iv)}]If $\sigma(x^{2n})$ is only divisible by Mersenne primes, then $2n \in \{2,4,6\}$.
More precisely, $\sigma(x^2)=M_1, \ \sigma(x^4) = M_4$ and $\sigma(x^6) = M_2M_3$.
\item[\rm{(v)}]If $\sigma(x^{2n})$ is a Mersenne prime, then $2n \in \{2,4\}$.
\item[\rm{(vi)}]If $\sigma(x^{n}) = \sigma((x+1)^{n})$,
then $n= 2^r-2$, for some $r \in \N^*$.
\item[\rm{(vii)}]If $\sigma(P^{2n}) = Q^m$, then $m =1$.
\end{itemize}
\end{lemma}
\newpage~\\
We shall use the following lemma to state Lemmas \ref{canaday-b-u-p-3} and \ref{canaday-b-u-p-4}.
\begin{lemma} [\cite{Gall-Rahav-up-allmersenne}, Lemma 2.4 and (the proof of) Lemma 2.5] \label{recall-mersenne}~\\ 
\begin{itemize}
\item[\rm{(i)}] Let $m \in \N^*$ and $M \in {\mathcal{M}}$ be such that $\sigma(M^{2m})$ has only Mersenne
primes as odd divisors,
then $2m =2$, $M \in \{M_2, M_3\}$ and $\sigma(M^{2m}) \in \{M_1M_4, M_1M_5\}$.
\item[\rm{(ii)}] Let $m \in \N^*$ and $M \in {\mathcal{M}}$ be such that $\sigma(M^{2m+1})$ has only Mersenne
primes as odd divisors,
then $2m+1 = 3 \cdot 2^{\alpha}-1$, for some $\alpha \in \N^*$ and $M \in \{M_2, M_3\}$.
All odd divisors of $\sigma(M^{2m+1})$ lie in $\{M_1, M_4, M_5\}$.
\end{itemize}
\end{lemma}

\begin{lemma}\label{gcdunitary}
Let $T$ be an irreducible polynomial over $\F_2$ and $k,l\in \N^*$.
Then, $\text{${\gcd}_u(T^k,T^l) = 1 \ ($resp. $T^k)$ if $k \not= l \ ($resp. $k=l)$.}$ In particular,
$$\text{${\gcd}_u(T^k,T^{2n-k}) = 1$ for $k \not=n$ and }
\text{${\gcd}_u(T^k,T^{2n+1-k}) = 1$ for  $k \leq 2n+1$.}$$
\end{lemma}
\begin{corollary}\label{aboutsigmastar2}
Let $T \in \F_2[x]$ be irreducible and $n \in \N^*$. Then
\begin{itemize}
\item[\rm{(i)}] $\sigma^{**}(T^{2n}) = (1+T)\sigma(T^n) \sigma(T^{n-1}), \ \sigma^{**}(T^{2n+1}) = (1+T) (\sigma(T^n))^2$.
\item[\rm{(ii)}]For any $c \in \N^*$, $\sigma^{**}(T^{c})$ is divisible by $1+T$, but not by $T$.
\end{itemize}
\end{corollary}
\begin{proof} (ii) follows from (i). For (i),
 we get by Lemma \ref{gcdunitary}:
$$\sigma^{**}(T^{2n}) = 1+T+\cdots+T^{n-1} + T^{n+1}+\cdots+T^{2n}.$$
So, $\sigma^{**}(T^{2n}) = (1+T^{n+1})\sigma(T^{n-1}) = (1+T)\sigma(T^n) \sigma(T^{n-1})$.
Moreover, $\sigma^{**}(T^{2n+1}) = \sigma(T^{2n+1}) = (1+T)(1+T+\cdots + T^{n-1})^2$.
\end{proof}

\begin{corollary}\label{nombreminimal}
If $A$ is a nonconstant b.u.p.
polynomial over $\F_2$, then $x(x+1)$ divides $A$, so that $A$ is even and $\omega(A) \geq 2$.
\end{corollary}
\begin{proof} Let $c \in \N^*$ and let $P$ be irreducible such that $P^c \| A$. If $P$ is odd, then
$x(x+1)$ divides $1+P$ which in turn, divides $\sigma^{**}(P^c)$. Hence, it divides $\sigma^{**}(A) = A$.
If $P = x$, then as above, $1+P=1+x$ divides $\sigma^{**}(A) = A$. So, $x(x+1)$ divides $A$. We apply a similar argument if $P=x+1$.
\end{proof}
\begin{corollary}\label{expressigmastar} Let $T \in \F_2[x]$ be irreducible and $r, \alpha, u \in \N^*$, with $u$ odd. Then
\begin{itemize}
\item[\rm{(i)}]$\sigma^{**}(T^{4r})=(1+T)^{2^{\alpha}} \sigma(T^{2r}) (\sigma(T^{u-1}))^{2^{\alpha}}$, $\gcd(\sigma(T^{2r}), \sigma(T^{u-1}))=1,$
if $2r = 2^{\alpha}u$.
\item[\rm{(ii)}]
$\sigma^{**}(T^{4r+2})=(1+T)^{2^{\alpha}} \sigma(T^{2r}) (\sigma(T^{u-1}))^{2^{\alpha}}$,\ $\gcd(\sigma(T^{2r}), \sigma(T^{u-1}))=1,$
if $2r+1 = 2^{\alpha}u-1$.
\item[\rm{(iii)}]
$\sigma^{**}(T^{2^{\alpha}u-1})=(1+T)^{2^{\alpha}-1} \cdot (\sigma(T^{u-1}))^{2^{\alpha}}$.
\end{itemize}
\end{corollary}
\begin{proof}
(i): We apply Corollary \ref{aboutsigmastar2} with $2n=4r$. One has:
$$\sigma(T^{2r-1}) = \sigma(T^{2^{\alpha}u-1}) = (1+T)^{2^{\alpha}-1} (\sigma(T^{u-1}))^{2^{\alpha}}.$$
Now, $\gcd(\sigma(T^{2r}), \sigma(T^{u-1}))=1$ since $\gcd(\sigma(T^{2r}), \sigma(T^{2r-1}))=1$ (Lemma \ref{coprime}) and $\sigma(T^{u-1})$ divides $\sigma(T^{2r-1})$.\\
(ii) and (iii): similar arguments with $2n= 4r+2$ (resp. $2n+1=2^{\alpha}u-1$).
\end{proof}
We shall use Table \ref{table2} (obtained from Corollary \ref{expressigmastar}, for some $\sigma^{**}(T^c)$).
\begin{table}[h]
\caption{Some formulas for $T \in \{x,x+1\} \cup {\mathcal{M}}$}~\\
\centering
\begin{tabular}{|l|}
\hline
\\
$\sigma^{**}(T^{2}) = (1+T)^2,\ \sigma^{**}(T^{4}) = (1+T)^2 \sigma(T^2), \
\sigma^{**}(T^{6}) = (1+T)^4 \sigma(T^2)$\\
\\
$\sigma^{**}(T^{8}) = (1+T)^4 \sigma(T^4),\ \sigma^{**}(T^{10}) = (1+T)^2 (\sigma(T^2))^2 \sigma(T^4)$\\
\\
$\sigma^{**}(T^{12}) = (1+T)^2 (\sigma(T^2))^2 \sigma(T^6),\
\sigma^{**}(T^{14}) = (1+T)^8 \sigma(T^6)$\\
\\
$\sigma^{**}(T^{2^{\alpha}-1}) = (1+T)^{2^{\alpha}-1}, \text{ for $\alpha \in \N^*$}$.\\
\\
\hline
\end{tabular}
\label{table2}
\end{table}

\begin{corollary}\label{splitcriteria}
Let $T \in \{x,x+1\} \cup {\mathcal{M}}$ and $c \in \N^*$. Then, $\sigma^{**}(T^{c})$ splits over $\F_2$ if and only if $(c=2$ or $c=2^{\alpha}-1$, for some $\alpha \in \N^*).$
\end{corollary}

\begin{corollary}\label{M2M3}
\begin{itemize}
\item[\rm{(i)}] For any $j \leq 5$, neither $M_2$ nor $M_3$ divides $\sigma^{**}(M_j^{h_j})$.
\item[\rm{(ii)}] $\sigma^{**}({M_2}^4) = x^2(x+1)^4 M_1M_5$ and $\sigma^{**}({M_3}^4) = x^4(x+1)^2 M_1M_4$.
\item[\rm{(iii)}] If $j \in \{1,4,5\}$ and $r \geq 2$, then $\sigma^{**}({M_j}^{2r})$ has a non Mersenne prime divisor.
\end{itemize}
\end{corollary}
\begin{proof} For (i) and (iii), see Lemma \ref{recall-mersenne}-i). For (ii), see Table \ref{table2}.
\end{proof}
\begin{lemma}\label{M2M3divisors}
If $M_2$ divides $\sigma^{**}(x^{a})$, then $a \in \{12,14,7\cdot 2^n-1: n \in \N^*\}$.
In this case, $M_3$ also divides $\sigma^{**}(x^{a})$.
\end{lemma}
\begin{proof}
We refer to Lemma \ref{complete2}-iv) and Corollary \ref{expressigmastar}.

If $a$ is even, then put $a = 4r = 2(2^{\alpha}u)$ or $a=4r+2=2(2^{\alpha}u-1)$, for some $r, \alpha, u \geq 1$, $u$ odd. We see that $M_2$ must divide $\sigma(x^{2r}) \sigma(x^{u-1})$. So $2r = 6$, $u-1 \not= 6$ since $\gcd(\sigma(x^{2r}), \sigma(x^{u-1})) = 1$. Hence, $M_3$ also divides $\sigma(x^{2r})$ and $a \in \{12,14\}$.

If $a=2^{n}u-1$ with $n,u \geq 1$, $u$ odd, then $M_2$ must divide $\sigma(x^{u-1})$. So, $u =7$ and $M_3$ also divides $\sigma(x^{u-1})$.
\end{proof}

Lemmas \ref{complete2} and \ref{recall-mersenne} imply the following four ones.
\begin{lemma}\label{canaday-b-u-p-1}
Let $m \in \N^*$ and $T \in \{x,x+1\}$ be such that the only odd prime divisors of $\sigma^{**}(T^{2m})$ are Mersenne primes, then $2m \in \{4,6,8,10,12,14\}$.
In this case, all its divisors lie in ${\mathcal{M}}$.
\end{lemma}

\begin{proof} It suffices to treat the case $T=x$ (for $T=x+1$, consider $\overline{T} = x$).

Case 1: $2m=4r$, with $r \geq 1$ and $2r = 2^{\alpha}u$, $u$ odd.
Corollary \ref{expressigmastar} gives: $\sigma^{**}(x^{2m}) = (1+x)^{2^{\alpha}} \cdot \sigma(x^{2r}) \cdot (\sigma(x^{u-1}))^{2^{\alpha}}$,
where $\sigma(x^{2r}) \sigma(x^{u-1})$ factors in $\{x,x+1\} \cup {\mathcal{M}}$.
Therefore, by Lemma \ref{complete2}-iv), $2r \in \{2,4,6\}$ and $u \in \{1,3,5,7\}$. So, $2m \in \{4,8,12\}$.

Case 2:  $2m=4r+2$, with $r \geq 1$ and $2r+1 = 2^{\alpha}u-1$, $u$ odd.
One has: $\sigma^{**}(x^{2m})=(1+x)^{2^{\alpha}} \cdot \sigma(x^{2r}) \cdot (\sigma(x^{u-1}))^{2^{\alpha}}$. As above, $2r \in \{2,4,6\}$ and $u \in \{1,3,5,7\}$, $2r = 2^{\alpha} u-2$. So, $2m \in \{6,10,14\}$.
It remains to remark that $\sigma(x^2)=M_1, \ \sigma(x^4)=M_4$ and $\sigma(x^6)=M_2 M_3$.
\end{proof}

\begin{lemma}\label{canaday-b-u-p-2}
Let $m \in \N^*$ and $T \in \{x,x+1\}$ be such that $\sigma^{**}(T^{2m+1})$ has only Mersenne primes as odd divisors, then
$2m+1 = 2^{\alpha} u -1$, for some $\alpha \in \N^*$ and $u \in \{3,5,7\}$.
In this case, all its odd divisors lie in ${\mathcal{M}}$.
\end{lemma}

\begin{proof}
Assume that $T = x$. One has: $\sigma^{**}(x^{2m+1}) = \sigma(x^{2m+1}) = (1+x)^{2^{\alpha}-1} (\sigma(x^{u-1}))^{2^{\alpha}}$. Therefore,
any odd divisor of $\sigma(x^{u-1})$ is a Mersenne prime. Thus, $u-1 \in \{2,4,6\}$ by Lemma \ref{complete2}-(iv).
\end{proof}

\begin{lemma}\label{canaday-b-u-p-3}
Let $M \in {\mathcal{M}}$ and $m \in \N^*$ be such that $\sigma^{**}(M^{2m})$ has only Mersenne primes as odd divisors,
then $2m \in \{4,6\}$ and $M \in \{M_2, M_3\}$.
In this case, all its divisors lie in $\{M_1, M_4, M_5\}$.
\end{lemma}
\begin{proof}
Case 1: $2m=4r$, with $r \geq 1$ and $2r = 2^{\alpha}u$, $u$ odd.
One has: $\sigma^{**}(M^{2m}) = (1+M)^{2^{\alpha}} \cdot \sigma(M^{2r}) \cdot (\sigma(M^{u-1}))^{2^{\alpha}}$.
So, $\sigma(M^{2r})$ and $\sigma(M^{u-1})$ (if $u \not= 1$) are only divisible by Mersenne primes and thus
$M \in \{M_2, M_3\}$ and $(2r=2$ or $u-1=2)$ (Lemma \ref{recall-mersenne}-i)). We must not have: $2r=u-1$ because
$\gcd(\sigma(M^{2r}), \sigma(M^{u-1}))=1$ (Corollary \ref{expressigmastar}). It follows that $2r = 2$ and $u=1$.
So, $\sigma(M^{2r}) = \sigma(M^2) \in \{M_1M_4, M_1M_5\}$.

Case 2:  $2m=4r+2$, with $r \geq 1$ and $2r+1 = 2^{\alpha}u-1$, $u$ odd.
One has: $\sigma^{**}(M^{2m})=\sigma(M^{2m})=(1+M)^{2^{\alpha}} \cdot \sigma(M^{2r}) \cdot (\sigma(M^{u-1}))^{2^{\alpha}}$.
As above, $M \in \{M_2,M_3\}$, $2r = 2$,  $u=1$ and $\sigma(M^{2r}) \in \{M_1M_4, M_1M_5\}$.
\end{proof}

\begin{lemma}\label{canaday-b-u-p-4}
Let $M \in {\mathcal{M}}$ and $m \in \N^*$ be such that $\sigma^{**}(M^{2m+1})$ has only Mersenne primes as odd divisors, then $2m+1 \in
\{3 \cdot 2^{\alpha}-1: \alpha \in \N^* \}$ and $M \in \{M_2, M_3\}$.
In this case, all its odd divisors lie in $\{M_1, M_4, M_5\}$.
\end{lemma}
\begin{proof}
Apply Lemma \ref{recall-mersenne}-(ii) since $\sigma^{**}(M^{2m+1}) = \sigma(M^{2m+1})$.
\end{proof}

\section{Proof of Theorem \ref{casemersenne}}\label{casmersenne}

Direct computations prove that our conditions are sufficient. For the necessities, we shall apply
Lemmas \ref{canaday-b-u-p-1}, \ref{canaday-b-u-p-2}, \ref{canaday-b-u-p-3} and \ref{canaday-b-u-p-4}. We fix:
\begin{itemize}
\item[ ] $\text{$\displaystyle{A= x^a(x+1)^b \prod_{i \in I}  P_i^{h_i}= A_1 A_2}$, where $a,b, h_i \in \N$,
$P_i$ is a Mersenne prime,}$
\item[ ] $\text{$\displaystyle{A_1 = x^a(x+1)^b \prod_{i=1}^5 M_i^{h_i}}$
and $\displaystyle{A_2 = \prod_{P_i \not\in {\mathcal{M}}} {P_i}^{h_i}}$.}$
\end{itemize}
We suppose that $A$ is {\it{i.b.u.p.}} so that $A_1 A_2 = A = \sigma^{**}(A) = \sigma^{**}(A_1) \sigma^{**}(A_2)$.

\subsection{First reduction}\label{firstreduction}

\begin{lemma}\label{gcdMjsigmA1}
\begin{itemize}
\item[\rm{(i)}] $\text{For any $P_j \not\in {\mathcal{M}}$, $\gcd(P_j^{h_j},\sigma^{**}(A_1)) = 1$ and $h_j = 0$}.$
\item[\rm{(ii)}] The polynomial $A_2$ equals $1$ and thus $A=A_1$.
\end{itemize}
\end{lemma}

\begin{proof}
By Lemmas \ref{canaday-b-u-p-1} and \ref{canaday-b-u-p-3}, any odd irreducible divisor of $\sigma^{**}(x^a)$
(resp. of $\sigma^{**}((x+1)^b)$, of $\sigma^{**}(M_i^{h_i})$, with $M_i \in {\mathcal{M}}$) must belong
to ${\mathcal{M}}$. Thus, for all $P_j \not\in {\mathcal{M}}$ and $M_i \in {\mathcal{M}}$, $P_j$ divides neither
$\sigma^{**}(x^a)$,
$\sigma^{**}((x+1)^b)$ nor $\sigma^{**}(M_i^{h_i})$. Hence, $\gcd(P_j^{h_j}, \sigma^{**}(A_1)) = 1$.
Moreover, $P_j^{h_j}$ divides $\sigma^{**}(A_2)$ because it divides
$A = \sigma^{**}(A)=\sigma^{**}(A_1) \sigma^{**}(A_2)$ and $\gcd(P_j^{h_j}, \sigma^{**}(A_1)) = 1$.
Hence, $A_2$ divides $\sigma^{**}(A_2)$.
So, by  Lemma \ref{sigmastardegree}, $A_2$ is b.u.p. and it is equal to $1$, $A$ being indecomposable.
\end{proof}

\begin{lemma}\label{divisorsigmA1}
If $Q$ is a Mersenne prime divisor of $\sigma^{**}(A_1)$, then $Q \in {\mathcal{M}}$.
\end{lemma}

\begin{proof} We apply Lemmas \ref{canaday-b-u-p-1} and \ref{canaday-b-u-p-3}.
If $Q$ divides $\sigma^{**}(x^a) \sigma^{**}((x+1)^b)$,
then $Q \in {\mathcal{M}}$. If $Q$ divides $\sigma^{**}(P_i^{h_i})$ with $P_i \in {\mathcal{M}}$,
then $P_i \in \{M_2, M_3\}$ and $Q \in \{M_1, M_4,M_5\}$.
\end{proof}

\begin{lemma}\label{hjvalues}
Suppose that $\sigma^{**}(M_j^{h_j})$ factors in $\{x, x+1\} \cup {\mathcal{M}}$, for some $j \leq 5$. Then
\begin{itemize}
\item[\rm{(i)}] $h_j \in \{2,2^n-1: n \in \N^*\}$ if $j \not\in  \{2, 3\}$,
\item[\rm{(ii)}] $h_j \in \{2,4,6\}$ or it is of the form $2^n u -1$, where $n\geq 1$ and $u \in\{1,3\}$,
if $j \in \{2,3 \}$.
\end{itemize}
\end{lemma}

\begin{proof}
These results are obtained from Corollary \ref{splitcriteria}, Lemma \ref{canaday-b-u-p-3} and Lemma \ref{canaday-b-u-p-4}.
\end{proof}

\begin{corollary}\label{caseperfect}
If $A_1$ is b.u.p., then $h_3=h_2$, $h_2 \in \{0,2,4,6,2^n-1, 3 \cdot 2^n-1: n \in \N^*\}$ and
$h_i \in \{0,2,2^n-1: n \in \N^*\}$, for  $i \in \{1,4,5\}$.
\end{corollary}

\begin{proof}
If $M_2$ (resp. $M_3$) divides $\sigma^{**}(A_1)$, then it divides $$V=\sigma^{**}(x^{a}) \sigma^{**}((x+1)^{b}).$$
Therefore, $M_3$ (resp. $M_2$) also divides $V$ and $\sigma^{**}(A_1)$. Hence, $h_2=h_3$.
Suppose that $h_j \geq 1$. The polynomial $\sigma^{**}(M_j^{h_j})$ must factor in $\{x, x+1\} \cup {\mathcal{M}}$.
Hence, by Lemma \ref{hjvalues}, $h_j \in \{2,2^n-1: n \in \N^*\}$, if $j \not\in  \{2, 3\}$. By the same lemma, for $j \in \{2,3 \}$, one has:
$h_j \in \{2,4,6\}$ or it is of the form $2^n u -1$, with $n\geq 1$, $u \in\{1,3\}$.
\end{proof}

In the rest of the paper, we prove the following proposition.
\begin{proposition}\label{A1-b-u-perfect}
If $A_1$ is b.u.p., then $A_1, \overline{A_1} \in \{C_1,\ldots, C_{23}\}$.
\end{proposition}

\subsection{Proof of Proposition \ref{A1-b-u-perfect}}\label{mainproof}

We write: $\displaystyle{A_1 = x^a(x+1)^b M_1^{h_1} M_2^{h_2} M_3^{h_3} M_4^{h_4} M_5^{h_5}}$.
Corollary \ref{caseperfect} implies that for any $i$, $h_i \in \{0,2,4,6,2^n-1, 3 \cdot 2^n-1: n \in \N^*\}$.

\begin{lemma}\label{aoubsup3}
For any $n,m \in \N^*$,  $a\not=2^n-1$ or $b\not= 2^m-1$.
\end{lemma}

\begin{proof}
If $a =2^n-1$ and $b= 2^m-1$ for some $n,m\geq 1$, then we get
$\sigma^{**}(x^a) = (x+1)^a$ and $\sigma^{**}((x+1)^b) = x^b$ (cf. Table \ref{table2}). Thus,
$$x^a(x+1)^b M_1^{h_1} \cdots M_5^{h_5} =A_1=\sigma^{**}(A_1)=
(x+1)^a x^b \sigma^{**}(M_1^{h_1}) \cdots \sigma^{**}(M_5^{h_5}).$$
Hence, $x^b$ and $(x+1)^a$ both divide $A_1$. Thus, $b \leq a \leq b$. So, $a = b$. It follows that $\displaystyle{M_1^{h_1} \cdots M_5^{h_5}}$ is b.u.p., which
contradicts Corollary \ref{nombreminimal}.
\end{proof}

By direct computations (sketched in Section \ref{compute}), we get Proposition \ref{A1-b-u-perfect}
from Lemmas \ref{mersenneaetb} and \ref{mersennabodd}.

Set $K_1=\{0,1,2,3,4,5,6,7,11,23\}$ and $K_2 = \{0,1,2,3,4,6,7,15\}$.

\begin{lemma}\label{mersenneaetb}
\begin{itemize}
\item[\rm{(i)}] If $a$ and $b$ are both even, then $a,b \leq 14$ and $h_i \in K_1$.
\item[\rm{(ii)}] If $a$ is even and $b$ odd, then $a \leq 14, b = 2^{\beta}v-1$, with $\beta \leq 3, v\leq 7$,
$v$ odd and $h_i \in K_1$.
More precisely,
$h_3= h_2,  h_2 \in \{0,2,4,6\}$ and $h_1,h_4, h_5 \in \{0,1,2,3,7,15\}$.
\end{itemize}
\end{lemma}

\begin{proof} According to Corollary \ref{caseperfect}, it remains to give upper bounds
for $a, b$ and for $h_i$, if $h_i$ is odd.

(i) If $a$ and $b$ are both even, then $a \ ($resp. $b)$ is of the form $4r$ or $4r+2$, $($resp. $4s$ or $4s+2)$.
Thus,
$\sigma(x^{2r})$ and $\sigma((x+1)^{2s})$ are both odd divisors
of $\sigma^{**}(A_1)=A_1$. Hence, they belong to ${\mathcal{M}}$. So, by Lemma \ref{complete2}-iv), one has $2r, 2s \leq 6$ and $a, b\leq 14$.
If $h_i$ is odd, then it is of the form $2^nu-1$, with $u \in \{1,3\}$ (Lemma \ref{hjvalues}). So,
$\sigma^{**}({M_i}^{h_i}) = \sigma({M_i}^{h_i}) = (1+M_i)^{2^n-1} (\sigma({M_i}^{u-1}))^{2^n}$.
Thus, $2^n-1 \leq a \leq 14$, by considering the exponents of $x$ in $A_1$ and in $\sigma^{**}(A_1)$.
We get $n \leq 3$ and $h_i \in K_1$.

(ii) In this case, $x^{2^{\beta} -1}$ divides $\sigma^{**}((x+1)^b)$ which in turn, divides $\sigma^{**}(A_1) = A_1$.
So, $2^{\beta} -1 \leq a$. Since $a \leq 14$ as we have seen in i), one has $\beta \leq 3$. Moreover,
$\sigma((x+1)^{v-1})$ lies in $\{1, M_1, M_2M_3, M_5\}$. We deduce (Lemma \ref{complete2}) that $v\leq 7$.
As above, $h_i \in K_1$.
\end{proof}

\begin{lemma}\label{mersennabodd}
If $a$ and $b$ are both odd, then $a = 2^{\alpha}u-1$, $b = 2^{\beta}v-1$ with $u,v \leq 7$,
$u, v$ both odd, $(u,v) \not= (1,1)$,
$1\leq \alpha, \beta \leq 3$ and $h_i \in K_2$.
\end{lemma}
\begin{proof} We give upper bounds
for $a, b$ and for $h_i$, if $h_i$ is odd.
One has
$$\text{$\sigma^{**}({x}^{a}) = (x+1)^{2^{\alpha}-1} (\sigma({x}^{u-1}))^{2^{\alpha}}$,
$\sigma^{**}({(x+1)}^{b}) = x^{2^{\beta}-1} (\sigma({(x+1)}^{v-1}))^{2^{\beta}}$.}$$
Without loss of generality, we may suppose that $u \leq v$.\\
$\bullet$ If $u=7$ or $v = 7$, then $h_2 \not= 0$ and $h_2 \in \{2^{\alpha}, 2^{\beta}, 2^{\alpha} + 2^{\beta}\}$
(compare $h_2$
with all possible exponents of $M_2$ in $\sigma^{**}(A_1)$).
So, $h_2$ is even and thus $h_3=h_2 \leq 6$, $\alpha, \beta \leq 2$ and $a,b \leq 27$.\\
Now, for $i \in \{1,4,5\}$ with $h_i$ odd,  one has: $h_i=2^n-1 \leq a \leq 27$, so that $n\leq 4$ and
$h_i \in \{1,3,7,15\}$.\\
$\bullet$ If $u,v \leq 5$, then $h_3=h_2 = 0$. For $j \in \{1,4,5\}$,  $M_j$ divides $\sigma^{**}(A_1)$ if
and only if it
divides $\sigma^{**}({x}^{a}) \sigma^{**}({(1+x)}^{b})$.

- The case $u=v=1$ does not happen because $A_1$ does not split.

- $\text{$h_1 = 2^{\beta} \leq 2$,
$2^{\alpha} - 1 \leq b=3 \cdot 2^{\beta}-1$, $\beta = 1$, $\alpha \leq 2$}$ if $u=1$ and $v=3$.

- $\text{$h_5 = 2^{\beta} \leq 2$,
$2^{\alpha} - 1 \leq b=5 \cdot 2^{\beta}-1$, $\beta = 1$, $\alpha \leq 3$}$ if $u=1$ and $v=5$.

- If $u=v=3$, then $h_1$ is even and $h_1 = 2^{\alpha} + 2^{\beta} \geq 4$, which is impossible.

- $h_1 = 2^{\alpha} \leq 2$, $h_5 = 2^{\beta} \leq 2$,
$\alpha=\beta = 1$  if $u=3$ and $v=5$.

- $h_4 = 2^{\alpha} \leq 2$, $h_5 = 2^{\beta} \leq 2$ and $\alpha=\beta = 1$ if $u=v=5$.
\end{proof}
\subsection{Maple Computations}\label{compute}
According to Lemmas \ref{mersenneaetb} and \ref{mersennabodd}, we determine, in 4 steps,
the set $L$ of all 7-uples $[a,b,h_1,h_2,h_3,h_4,h_5]$
such that $a \leq b$. Then, we search $S=\displaystyle{x^a(x+1)^b {M_1}^{h_1} {M_2}^{h_2} {M_3}^{h_3} {M_4}^{h_4}
{M_5}^{h_5}}$ satisfying: $\sigma^{**}(S)= S$.\\
The case where $a \geq b$ is obtained from the substitution: $x \longleftrightarrow x+1$.\\
1) If $a$ and $b$ are even, then $b \in \{0,2,4,6,8,10,12,14\}$ and
$h_i \in K_1$.\\
2) If $a$ is even and $b$ odd, then $b \in \{1, 3, 5, 7, 9, 11, 13, 19, 23, 27, 39, 55\}$,
$a \in \{0,2,4,6,8,10,12,14\}$ and $h_i \in K_1$.\\
3) If $a$ is odd and $b$ even, then $a \in \{1, 3, 5, 7, 9, 11, 13, 19, 23, 27, 39, 55\}$,
$b \in \{0,2,4,6,8,10,12,14\}$ and $h_i \in K_1$.\\
4) If $a$ and $b$ are odd, then $a, b \leq 27$ and
$h_i \in K_2$, $h_2=h_3 = 0$ if $u,v \leq 5$.\\
\\
We get the following results.
\begin{center}
\begin{tabular}{|l|l|l|l|}
\hline
Case&{card}(L)& b.u.p. polynomials&Calculation time\\
\hline
&&&\\
1)&$35000$&$C_3, C_4, C_8, C_{13}, C_{14}, C_{15}$& 6 min\\
\hline
&&&\\
2)&$70000$&$C_5, C_9, C_{16}, C_{18}, C_{20}$& 15 min\\
\hline
&&&\\
3)&$35000$&$C_1, C_6, C_{10}, C_{21}, C_{22}, C_{23}$& 6 min\\
\hline
&&&\\
4)&$97500$&$C_2, C_7, C_{11}, C_{12}, C_{17}, C_{19}$& 30 min\\
\hline
\end{tabular}
\end{center}~\\
{\bf{The function $\sigma^{**}$ is defined as}} Sigm2star in the {\tt{Maple}} code.
\begin{verbatim}
> Sigm2star1:=proc(S,a) if a=0 then 1;else if a mod 2 = 0
then n:=a/2:sig1:=sum(S^l,l=0..n):sig2:=sum(S^l,l=0..n-1):
Factor((1+S)*sig1*sig2) mod 2:
else Factor(sum(S^l,l=0..a)) mod 2:fi:fi:end:
> Sigm2star:=proc(S) P:=1:L:=Factors(S) mod 2:k:=nops(L[2]):
for j to k do S1:=L[2][j][1]:h1:=L[2][j][2]:
P:=P*Sigm2star1(S1,h1):od:P:end:
\end{verbatim}
\section{Conjecture}
We remark that for $A \in \{C_j: 1\leq j \leq 13\}$, one has $\omega(A) \in \{3,4\}$. After several computations, we may establish
\begin{conjecture}
Let $A \in \F_2[x]$ be such that $A$ does not split and $\omega(A) \leq 4$. Then, $A$ is b.u.p if and only if
$A \in \{\{C_j: 1\leq j \leq 13\} \cup \{D_1,D_2\}$, where $D_1 = x^4(x+1)^5 (1+x+x^2)(1+x+x^4), D_2 = x^4(x+1)^5 (1+x+x^2)^5(1+x+x^4)^2$.
\end{conjecture}
Note that both $D_1$ and $D_2$ are divisible by the non Mersenne prime $S_1=1+x+x^4$. In the future work, we shall try to prove this conjecture.

\end{document}